\documentclass[final]{siamltex}
\usepackage[pagewise]{lineno}
\usepackage{mathrsfs}
\usepackage{amsmath}
  \usepackage{paralist}
  \usepackage{graphics} %% add this and next lines if pictures should be in esp format
  \usepackage{epsfig} %For pictures: screened artwork should be set up with an 85 or 100 line screen
\usepackage{graphicx}  \usepackage{epstopdf}%This is to transfer .eps figure to .pdf figure; please compile your paper using PDFLeTex or PDFTeXify.
 \usepackage[colorlinks=true]{hyperref}
\makeatother

\normalsize
\usepackage{amsfonts}
\usepackage{arydshln}
\numberwithin{equation}{section}

\newtheorem{remark}[theorem]{Remark}

\newcommand{\lbl}[1]{\label{#1}}
\newcommand\ee{\end{equation}}
\newcommand\bes{\begin{eqnarray}}
\newcommand\ees{\end{eqnarray}}
\newcommand\bess{\begin{eqnarray*}}
\newcommand\eess{\end{eqnarray*}}

\title{Pointwise wave behavior of the Navier-Stokes equations in half space}

\author{LINGLONG DU\footnotemark[2]
\ \ Haitao Wang \footnotemark[3] % \footnotemark[4]
        }

\begin{document}

\maketitle
\renewcommand{\thefootnote}{\fnsymbol{footnote}}
 \footnotetext[2]{Department of Applied Mathematics, Donghua University,
Shanghai, P.R. China, ({\tt matdl@dhu.edu.cn}). }
\footnotetext[3]{Institute of Natural Sciences and School of Mathematical Sciences,
	Shanghai Jiao Tong University, Shanghai, P.R. China, ({\tt haitaowang.math@gmail.com, Corresponding author}).}
% \footnotetext[4]{Institute of Mathematics, Academia Sinica, Taipei, Taiwan, ({\tt haitaowang.math@gmail.com, Corresponding author}). }
 
\renewcommand{\thefootnote}{\arabic{footnote}}

\begin{abstract}

In this paper, we investigate the pointwise behavior of the solution for the compressible Navier-Stokes equations with mixed boundary condition in half space. Our results show that the leading order of  Green's function  for the linear system in half space are heat kernels propagating with sound speed in two opposite directions and reflected heat kernel (due to the boundary effect) propagating with positive sound speed. With the strong wave interactions, the nonlinear analysis exhibits the rich  wave structure: the diffusion waves interact with each other and consequently, the solution decays with algebraic rate.
  \end{abstract}

\begin{keywords}
mixed boundary condition,  Green's function,  nonlinear wave interaction
\end{keywords}

\begin{AMS}
76N10, 35B40, 35A08
\end{AMS}

\pagestyle{myheadings} \thispagestyle{plain} \markboth{LINGLONG DU and HAITAO WANG}{Long time wave behavior of the Navier-Stokes equations in half space}

\section{Introduction}

The Navier-Stokes  equations are the  fundamental system in the fluid dynamics.  Qualitative and quantitative  studies on the  fluid dynamic theory will help us to understand the physical phenomena  much better. 
Most of the interesting phenomena in the fluid dynamics are related to the presence of a physical boundary, such as  slip boundary layer,  thermal creep flow, and curvature effects. They can be understood only with the knowledge of the interaction between the fluid waves and the boundary layer.  In this paper, we begin to study the  one dimensional  isentropic Navier-Stokes equations.  We focus on the structure of long time solution  in the sense of pointwise wave propagation. There have been some essential progress in developing the pointwise estimate of long time solution for the initial-boundary value problem of different models such as   \cite{D, D1, Du,  LL, LY2, WY} and the references therein. 

The isentropic Navier-Stokes equations in Eulerian coordinate with mixed boundary condition in half line are
\begin{equation}
\begin{cases}
\rho_{t}+m_{x}=0,\ \  x\in \mathbb{R}_+,\ \ t>0, \\
m_{t}+\left(\frac{m^{2}}{\rho}+p\left(\rho\right)\right)_{x}=\nu\left(\frac{m}{\rho}\right)_{xx},\\
\left.\left(a_{1}\partial_{x}m+a_{2}m\right)\right|_{x=0}=0,
\end{cases}\label{eq:NL_prob*}
\end{equation}
where $\rho(x,t)>0$ is the density, $m(x,t)$ represents the momentum, and $p(\rho)$ represents the pressure.  $\nu$ is the viscosity, $a_1$ and $a_2$ are constants. Here we only treat the case  $a_1a_2<0$. The boundary condition of  Dirichlet type ($a_1=0$) and  Neumann type ($a_2=0$) are much simpler.  For the case of  $a_1a_2>0$, the linearized system is unstable. These are  explained in Section 2.   To overcome the complexity of the mixed type condition, we first construct Green's function for the linearized initial-boundary value problem.
 
We linearize the above system around $\rho=1$ and $m=0$, and denote
the perturbation still by $\rho$ and $m$:
\begin{equation}
\begin{cases}
\rho_{t}+m_{x}=0,\\
m_{t}+c^{2}\rho_{x}=\nu m_{xx}+Q,\\
\left.\left(a_{1}\partial_{x}m+a_{2}m\right)\right|_{x=0}=0,
\end{cases}\label{eq:NL_Cauchy}
\end{equation}
where $c=\sqrt{p'\left(1\right)}$ is the sound speed, and the nonlinear term
\[
Q\equiv\tilde{Q}_x=-\left[\frac{m^{2}}{1+\rho}+p\left(1+\rho\right)-p(1)-p'\left(1\right)\rho+\nu\left(\frac{\rho m}{1+\rho}\right)_{x}\right]_x.
\]

In \cite{KN}, the authors proved the global existence  for the initial-boundary value problem of compressible Navier-Stokes equations in three-dimensional half space and showed the convergence of the solution to the equilibrium state by using the energy estimates.  Later, \cite{KK,KK1} obtained  the decay rate of the solution for the half space  $\mathbb{R}^n_+$  for $n\ge 2$ with Dirichet boundary condition. They proved that if the initial perturbation of constant state is in $H^{s+l}\cap L^1$, where $s=[\frac{n}{2}]+1$ and $l$ is a nonnegative integer, then the $L^p$ norm solution of linearized problem has the optimal $O(t^{-\frac{n}{2}(1-\frac{1}{p})})$ decay rate with $p\in[2,\infty]$. Here we refine the $L^p$ estimates into more accurate pointwise estimates and give a better understanding of local nonlinear wave interactions in one dimensional case. For the high dimensional case, we leave it to the future, however.

The Green's function for the linearized initial-boundary  value problem satisfies the following equation system
 \bes
 \left\{\begin{array}{lll}\medskip
\left(\partial_t+\left( \begin{array}{ccc}\medskip  0 &1\\
 c^2&0 \end{array}\right)\partial_x-\left( \begin{array}{ccc}\medskip  0 &0\\
 0 &\nu \end{array}\right)\partial_{xx}\right)\mathbb{G}(x,t;y)=0,x>0,y>0, t>0,\\
\mathbb{G}(x,0;y)=\delta(x-y)I,\\
\begin{pmatrix}-a_{1}\partial_{t} & a_{2}\end{pmatrix}\mathbb{G}(0,t;y)=0. \end{array}\right.\lbl{1.8}
 \ees
Here $I$ is $2\times2$ identity matrix and $\delta(x-y)$ is the Dirac function. The boundary condition has
been rewritten in view of $\partial_{x}m=-\partial_{t}\rho$. 

Our first result is about the pointwise structure of  the Green's function  $\mathbb{G}(x,t;y)$:
      \begin{theorem}\lbl{l b}
      When $a_1a_2<0$, or $a_1=0$, or $a_2=0$, there exists a constant $C$ such that the Green's function $\mathbb{G}(x,t;y)$ of  linearized system (\ref{1.8})  has the following estimates for all $0\leq x,y<\infty$, $t\geq0$:
      \bes
    &&\left|D_x^{\alpha}\left(\mathbb{G}(x,t;y)-e^{-\frac{c^2t}{2\nu}}(\delta(x-y)-\delta(x+y))\left( \begin{array}{ccc}\medskip 1 &0\\
 0 &0 \end{array}\right)\right)\right|\nonumber\\
    &\le& O(1)t^{-\alpha/2}\left(\frac{e^{-\frac{(x-y+ct)^2}{2\nu t}}}{\sqrt{\nu t}}+\frac{e^{-\frac{(x-y-ct)^2}{2\nu t}}}{\sqrt{\nu t}}
 +\frac{e^{-\frac{(x+y-ct)^2}{(2\nu+\varepsilon) t}}}{\sqrt{\nu t}}\right)\nonumber\\
 & & +O(1)e^{-(|x-y|+t)/C}+O(1)e^{-(|x+y|+t)/C}.\lbl{1.8a}
 \ees
Here $\alpha$ is a non-negative integer.
 \end{theorem}

 The pointwise estimate for the Green's function allows us to give the time-asymptotic behavior of the solution to the nonlinear initial-boundary value problem:
        \begin{theorem}\lbl{LC}
Assume that the initial data $\left(\rho_{0},m_{0}\right)$
satisfying $\left\Vert \left(\rho_{0}-1,m_{0}\right)\right\Vert _{H^{4}\left(\mathbb{R}_{+}\right)}\leq\varepsilon_{0}$
and 
\bess
\left|D_x^{\alpha}\left(\rho_{0}-1,m_{0}\right)\left(x\right)\right|=O(1)\varepsilon_{0}(1+x^2)^{-r/2},\ \ r>1/2,
\eess  for $\alpha\le 1$.
Then for $\varepsilon_{0}$ sufficiently small, there exists a unique
global classical solution to the problem \eqref{eq:NL_prob*} with
the mixed boundary condition where $a_{1}a_{2}<0$ or $a_1=0$, or $a_2=0$. Moreover, the
solution has the following pointwise estimates
 for $\alpha\le 1$, 
\bess
\left|D_x^{\alpha}\left(\rho-1,m\right)\left(x,t\right)\right|
\leq O(1)\varepsilon_{0}(1+t)^{-\alpha/4}\left[\left(x-c\left(t+1\right)\right)^{2}+\left(t+1\right)\right]^{-1/2},
\eess
here $c=\sqrt{p'\left(1\right)}$ is the sound speed.
\end{theorem}

\begin{remark}
Comparing  to \cite{KK1},  we get the optimal pointwise estimate for $\rho$ and  $m$.  For $\alpha=1$,  we  only get  extra $(1+t)^{-1/4}$ time decay rate for the solution, this is due to the effects of  boundary  and  closure of the nonlinear interactions. There still leaves room for improvement. We will investigate it further in the future. 
\end{remark}

  \begin{corollary}Under the assumptions in Theorem \ref{LC},
we have the following optimal $L^p(\mathbb{R}^+)$ estimates of the solution
\bess
\left\|(\rho-1,m)(\cdot,t)\right\|_{L^p(\mathbb{R}^+)}\leq
O(1)\varepsilon_0(1+t)^{-\frac{1}{2}(1-\frac{1}{p})},\
\ p\in(1,\infty].
\eess

\end{corollary}
The proof of the nonlinear estimates is based on the pointwise description of Green's function for the linearized initial-boundary value problem and the  Duhamel's principle.  The nonlinearity in one dimensional problem  is  much stronger than that in higher dimensional cases, so we
need more accurate Green's function  to identify the exact leading order.  The leading term of the fundamental solution for the linearized Cauchy problem is the convective heat kernel. With the help of the connection between fundamental solution and Green's function in the transformed space,  extra mirror term is obtained due to the boundary effect for Green's function. Thus the leading waves  of Green's function propagate in two directions: one  are heat kernel and reflected heat kernel propagating with positive sound speed; the other is the heat kernel propagating with negative sound speed.  Due to the strong nonlinear interaction of these waves with source terms and  nonlinear terms, the solution decays with  algebraic rate.
  
  Throughout this paper we use $O(1)$, $C$, $D_0$, $D$, $\cdots$ to denote  universal positive constants.  Denote by $L^p$ and $W^{m,p}$ the usual Lebesgue and Sobolev spaces on $\mathbb{R}_+$ and $H^m=W^{m,2}$, with norms $\|\cdot\|_{L^p}, \|\cdot\|_{W^{m,p}},\ \|\cdot\|_{H^m}$, respectively.

  The rest of paper is arranged as follows: In Section 2, we construct Green's function for the initial-boundary value problem.  The  pointwise structures of the nonlinear solution are studied  in Section 3.  Section 4 contains some useful lemmas for the nonlinear coupling of different diffusion waves.

 \section{The Green's function for the initial-boundary value problem}
 
 \subsection{Fundamental solution for the Cauchy problem}
 The necessary preliminary for the construction of Green's function is the fundamental
solution of Cauchy problem. The fundamental solution for linearized isentropic Navier-Stokes equations solves the system
\begin{equation}
\begin{cases}
\left(\partial_{t}+\begin{pmatrix}0 & 1\\
c^{2} & 0
\end{pmatrix}\partial_{x}-\begin{pmatrix}0 & 0\\
0 & \nu
\end{pmatrix}\partial_{x}^{2}\right)G\left(x,t\right)=0, & x\in\mathbb{R},\:t>0,\\
G\left(x,0\right)=\delta\left(x\right)I.
\end{cases}\label{eq:LG}
\end{equation}

The pointwise estimate of  $G(x,t)$ is  studied by \cite{LZ, Z}. It can be estimated by studying the inverse Fourier transform of $ \mathcal{F}[G](\xi,t)$
        \bes
       \mathcal{F}[G](\xi,t)&&=\left( \begin{array}{ccc}\medskip \frac{\sigma_+e^{\sigma_-t}-\sigma_-e^{\sigma_+t}}{\sigma_+-\sigma_-}\ \  &-i\xi\frac{e^{\sigma_+t}-e^{\sigma_-t}}{\sigma_+-\sigma_-}\\
-i\xi c^2\frac{e^{\sigma_+t}-e^{\sigma_-t}}{\sigma_+-\sigma_-}\ \  & \frac{\sigma_+e^{\sigma_+t}-\sigma_-e^{\sigma_-t}}{\sigma_+-\sigma_-} \end{array}\right),\lbl{1.7}
        \ees
        where 
        \bess\sigma_{\pm}=-\frac{1}{2}\xi(\nu\xi\pm\sqrt{\nu^2\xi^2-4c^2}).\eess
    
Therefore, we have
    \bess
    G(x,t)&=&e^{-\frac{c^2t}{\nu}}\delta(x)\left( \begin{array}{ccc}\medskip 1 &0\\
 0 &0 \end{array}\right)+\frac{e^{-\frac{(x+ct)^2}{2\nu t}}}{\sqrt{\nu t}}\left( \begin{array}{ccc}\medskip \frac{1}{2} &\frac{1}{2c}\\
 \frac{c}{2} &\frac{1}{2} \end{array}\right)+\frac{e^{-\frac{(x-ct)^2}{2\nu t}}}{\sqrt{\nu t}}\left( \begin{array}{ccc}\medskip \frac{1}{2} &-\frac{1}{2c}\\
 -\frac{c}{2} &\frac{1}{2} \end{array}\right)\\
 &&+O(1)(t+1)^{-\frac{1}{2}}t^{-\frac{1}{2}}[e^{-\frac{(x+ct)^2}{Dt}}+e^{-\frac{(x-ct)^2}{Dt}}]+O(1)e^{-(|x|+t)/C},
    \eess
where  $D>0$ is a constant.
            
\subsection{The connection between Green's function and fundamental solution}
Introduce the Fourier transform and Laplace transform for $f(x,t)$:
\begin{eqnarray}
 \mathcal{F}[f](\xi,t)=\int_{ \mathbb{R}}e^{-i\xi x}f(x,t)dx,\\
  \mathcal{L}[f](x,s)=\int_0^{\infty}e^{-st}f(x,t)dt.
\end{eqnarray}

It turns out that the fundamental solution and the Green's function are
closely related in the transformed variables. Taking Laplace transform
in $t$ and Fourier transform in $x$ to the first equation in \eqref{eq:LG}, denoting
the transformed variables by $s$ and $\xi$ respectively, and representing
the transformed  one by $\mathcal{F} [\mathcal{L}[G]]\left(\xi,s\right)$, we obtain 
\[
\begin{pmatrix}s & i\xi\\
ic^{2}\xi & s+\nu\xi^{2}
\end{pmatrix}\mathcal{F} [\mathcal{L}[G]]\left(\xi,s\right)=\begin{pmatrix}1 & 0\\
0 & 1
\end{pmatrix}.
\]
Invert the matrix to yield
\[
\mathcal{F} [\mathcal{L}[G]]\left(\xi,s\right)=\frac{1}{s^{2}+\left(c^{2}+\nu s\right)\xi^{2}}\begin{pmatrix}s+\nu\xi^{2} & -i\xi\\
-ic^{2}\xi & s
\end{pmatrix}.
\]
Taking inverse Fourier transform to $\xi$ and using
\[
\frac{1}{2\pi}\int_{\mathbb{R}}\frac{e^{i\xi x}}{\xi^{2}+\frac{s^{2}}{c^{2}+\nu s}}d\xi=\frac{e^{-\lambda\left|x\right|}}{2\lambda},
\]
where $\lambda=\lambda\left(s\right)=s/\sqrt{\nu s+c^{2}}$,
we have
\[
\mathcal{L}[G]\left(x,s\right)=\frac{1}{\nu s+c^{2}}\begin{pmatrix}\nu\delta\left(x\right)+\frac{c^{2}s}{\nu s+c^{2}}\frac{e^{-\lambda\left|x\right|}}{2\lambda} & \frac{\mbox{sgn}\left(x\right)e^{-\lambda\left|x\right|}}{2}\\
\frac{c^{2}\mbox{sgn}\left(x\right)e^{-\lambda\left|x\right|}}{2} & \frac{e^{-\lambda\left|x\right|}}{2\lambda}s
\end{pmatrix}\,.
\]
In particular,
when $\bar{x}>0$,
\[
\mathcal{L}[G]\left(-\bar{x},s\right)=\frac{1}{\nu s+c^{2}}\begin{pmatrix}\frac{c^{2}s}{\nu s+c^{2}}\frac{1}{2\lambda} & -\frac{1}{2}\\
-\frac{c^{2}}{2} & \frac{1}{2\lambda}s
\end{pmatrix}e^{-\lambda\bar{x}}.
\]

Now we consider the initial-boundary value problem, which is the linearization of \eqref{eq:NL_Cauchy},
\begin{equation}\label{linear_prob}
\begin{cases}
\rho_{t}+m_{x}=0,\quad x>0,\:t>0,\\
m_{t}+c^{2}\rho_{x}=\nu m_{xx},\\
\begin{pmatrix}\rho\left(x,0\right) & m\left(x,0\right)\end{pmatrix}=\begin{pmatrix}\rho_{0}\left(x\right) & m_{0}\left(x\right)\end{pmatrix},\\
\left.\left(-a_{1}\rho_{t}+a_{2}m\right)\right|_{x=0}\equiv0.
\end{cases}
\end{equation}
Making use of the fundamental solution to take care the
initial data,
\[
\begin{pmatrix}\bar{\rho}\\
\bar{m}
\end{pmatrix}\left(x,t\right)\equiv\int_{0}^{\infty}G\left(x-y,t\right)\begin{pmatrix}\rho_{0}\left(y\right)\\
m_{0}\left(y\right)
\end{pmatrix}dy,
\]
then the functions $\left(\rho-\bar{\rho},\:m-\bar{m}\right)$, which
will be  denoted by $\left(\tilde{\rho},\:\tilde{m}\right)$, satisfy
the following homogeneous initial data problem
\[
\begin{cases}
\tilde{\rho}_{t}+\tilde{m}_{x}=0,\quad x>0,~t>0,\\
\tilde{m}_{t}+c^{2}\tilde{\rho}_{x}=\nu\tilde{m}_{xx},\\
\begin{pmatrix}\tilde{\rho}\left(x,0\right) & \tilde{m}\left(x,0\right)\end{pmatrix}=0,\\
\left.\left(-a_{1}\tilde{\rho}_{t}+a_{2}\tilde{m}\right)\right|_{x=0}=-\begin{pmatrix}-a_{1}\partial_{t} & a_{2}\end{pmatrix}\int_{0}^{\infty}G\left(-y,t\right)\begin{pmatrix}\rho_{0}\left(y\right)\\
m_{0}\left(y\right)
\end{pmatrix}dy\equiv b\left(t\right).
\end{cases}
\]
Differentiating the first equation with respect to $t$, the second
equation with respect to $x$, and by suitable combinations we find that
\[
\tilde{\rho}_{tt}-c^{2}\tilde{\rho}_{xx}=\nu\tilde{\rho}_{txx},\quad\tilde{m}_{tt}-c^{2}\tilde{m}_{xx}=\nu\tilde{m}_{txx}.
\]
Taking Laplace transform in $t$ and using homogenous initial data to
get
\[
\mathcal{L}[\tilde{\rho}]_{xx}=\frac{s^{2}}{\nu s+c^{2}}\mathcal{L}[\tilde{\rho}],\qquad\mathcal{L}[\tilde{m}]_{xx}=\frac{s^{2}}{\nu s+c^{2}}\mathcal{L}[\tilde{m}].
\]
Solving the ODE and dropping out the divergent mode as $x\to+\infty$,
we have
\[
\mathcal{L}[\tilde{\rho}]\left(x,s\right)=\mathcal{L}[\tilde{\rho}_{b}]\left(s\right)e^{-\lambda x},\quad\mathcal{L}[\tilde{m}]\left(x,s\right)=\mathcal{L}[\tilde{m}_{b}]\left(s\right)e^{-\lambda x},
\]
where $\lambda(s)=s/\sqrt{\nu s+c^{2}}$ is defined as before, $\tilde{\rho}_{b},~\tilde{m}_{b}$ represent the Dirichlet data. On the other
hand, from the first equation of \eqref{linear_prob} and boundary relationship we have
\[
s\mathcal{L}[\tilde{\rho}_{b}]-\lambda\mathcal{L}[\tilde{m}_{b}]=0,\quad-a_{1}s\mathcal{L}[\tilde{\rho}_{b}]+a_{2}\mathcal{L}[\tilde{m}_{b}]=\mathcal{L}[b](s),
\]
which imply that 
\[
\mathcal{L}[\tilde{\rho}_{b}]=\frac{\lambda\mathcal{L}[ b](s)}{s\left(a_{2}-a_{1}\lambda\right)},\quad\mathcal{L}[\tilde{m}_{b}]=\frac{\mathcal{L}[b](s)}{a_{2}-a_{1}\lambda},
\]
hence
\[
\begin{pmatrix}\mathcal{L}[\tilde{\rho}]\\
\mathcal{L}[\tilde{m}]
\end{pmatrix}\left(x,s\right)=\frac{1}{a_{2}-a_{1}\lambda}\begin{pmatrix}\frac{\lambda}{s}\\
1
\end{pmatrix}\mathcal{L}[b]\left(s\right)e^{-\lambda x}.
\]
Thus the solution 
\begin{align*}
\begin{pmatrix}\mathcal{L}[\rho]\\
\mathcal{L}[m]
\end{pmatrix}\left(x,s\right) & =\begin{pmatrix}\mathcal{L}[\bar{\rho}]\\
\mathcal{L}[\bar{m}]
\end{pmatrix}\left(x,s\right)+\begin{pmatrix}\mathcal{L}[\tilde{\rho}]\\
\mathcal{L}[\tilde{m}]
\end{pmatrix}\left(x,s\right)\\
 & =\int_{0}^{\infty}\left[\mathcal{L}[G]\left(x-y,s\right)+\frac{1}{a_{2}-a_{1}\lambda}\begin{pmatrix}\frac{\lambda}{s}\\
1
\end{pmatrix}\begin{pmatrix}a_{1}s & -a_{2}\end{pmatrix}\mathcal{L}[G]\left(-y,s\right)e^{-\lambda x}\right]\begin{pmatrix}\rho_{0}\left(y\right)\\
m_{0}\left(y\right)
\end{pmatrix}dy.
\end{align*}

Therefore the transformed Green's function $\mathcal{L}[\mathbb{G}]\left(x,s;y\right)$ is
\[
\mathcal{L}[\mathbb{G}]\left(x,s;y\right)=\mathcal{L}[G]\left(x-y,s\right)+\frac{1}{a_{2}-a_{1}\lambda}\begin{pmatrix}\frac{\lambda}{s}\\
1
\end{pmatrix}\begin{pmatrix}a_{1}s & -a_{2}\end{pmatrix}\mathcal{L}[G]\left(-y,s\right)e^{-\lambda x}.
\]
Direct calculation shows that the second term in $\mathcal{L}[\mathbb{G}]\left(x,s;y\right)$
is
\bess
&&\frac{1}{a_{2}-a_{1}\lambda}\begin{pmatrix}a_{1}\lambda & -\frac{a_{2}\lambda}{s}\\
a_{1}s & -a_{2}
\end{pmatrix}\frac{1}{\nu s+c^{2}}\begin{pmatrix}\frac{c^{2}s}{\nu s+c^{2}}\frac{1}{2\lambda} & -\frac{1}{2}\\
-\frac{c^{2}}{2} & \frac{1}{2\lambda}s
\end{pmatrix}e^{-\lambda\left(x+y\right)}\\
&=&\frac{a_{2}+a_{1}\lambda}{a_{2}-a_{1}\lambda}\frac{1}{2}\begin{pmatrix}\frac{c^{2}}{\left(\nu s+c^{2}\right)^{\frac{3}{2}}} & -\frac{1}{\nu s+c^{2}}\\
\frac{c^{2}}{\nu s+c^{2}} & -\frac{1}{\sqrt{\nu s+c^{2}}}
\end{pmatrix}e^{-\lambda\left(x+y\right)}.
\eess
On the other hand we notice that for $x>0$ and $y>0$
\begin{align*}
\mathcal{L}[G]\left(x+y,s\right) & =\frac{1}{\nu s+c^{2}}\begin{pmatrix}\frac{c^{2}s}{\nu s+c^{2}}\frac{1}{2\lambda} & \frac{1}{2}\\
\frac{c^{2}}{2} & \frac{1}{2\lambda}s
\end{pmatrix}e^{-\lambda\left(x+y\right)}\\
 & =\frac{1}{2}\begin{pmatrix}\frac{c^{2}}{\left(\nu s+c^{2}\right)^{\frac{3}{2}}} & \frac{1}{\nu s+c^{2}}\\
\frac{c^{2}}{\nu s+c^{2}} & \frac{1}{\sqrt{\nu s+c^{2}}}
\end{pmatrix}e^{-\lambda\left(x+y\right)}.
\end{align*}
Therefore the Green's function in transformed variable is
\[
\mathcal{L}[\mathbb{G}]\left(x,s;y\right)=\mathcal{L}[G]\left(x-y,s\right)+\frac{a_{2}+a_{1}\lambda}{a_{2}-a_{1}\lambda}\mathcal{L}[G]\left(x+y,s\right)\begin{pmatrix}1 & 0\\
0 & -1
\end{pmatrix}.
\]
In summary,  the solution of the system \eqref{linear_prob} is
\[
\begin{pmatrix}\rho\\
m
\end{pmatrix}\left(x,t\right)=\int_{0}^{\infty}\left[G\left(x-y,t\right)+\mathcal{L}^{-1}\left[\frac{a_{2}+a_{1}\lambda}{a_{2}-a_{1}\lambda}\right]\underset{t}{\ast}G\left(x+y,t\right)\begin{pmatrix}1 & 0\\
0 & -1
\end{pmatrix}\right]\begin{pmatrix}\rho_{0}\left(y\right)\\
m_{0}\left(y\right)
\end{pmatrix}dy.
\]

For $a_{1}a_{2}>0$, it is easy to see that $\frac{a_{2}+a_{1}\lambda\left(s\right)}{a_{2}-a_{1}\lambda\left(s\right)}$
has a pole in right  half complex plane, which leads to an exponential growth
in time. In particular, for Dirichlet boundary condition $a_{1}=0$,
we have
\[
\mathbb{G}\left(x,t;y\right)=G\left(x-y,t\right)+G\left(x+y,t\right)\begin{pmatrix}1 & 0\\
0 & -1
\end{pmatrix};
\]
for Neumann boundary condition $a_{2}=0$, we have
\[
\mathbb{G}\left(x,t;y\right)=G\left(x-y,t\right)-G\left(x+y,t\right)\begin{pmatrix}1 & 0\\
0 & -1
\end{pmatrix}.
\]

Let us focus on the nontrivial case that $a_{1}a_{2}<0.$ 
Taking inverse Laplace  transform with respect to $t$, we have
\[
\mathbb{G}\left(x,t;y\right)=G\left(x-y,t\right)+G_{mir}\left(x+y,t\right).
\]
Here the subscript ``mir'' stands for ``mirror'' since this part
is analogous to mirror image of original fundamental solution,
\[
G_{mir}\left(x+y, t\right)\equiv\mathcal{L}^{-1}\left[\frac{a_{2}+a_{1}\lambda}{a_{2}-a_{1}\lambda}\mathcal{L}\left[G\right]\left(x+y,s\right)\begin{pmatrix}1 & 0\\
0 & -1
\end{pmatrix}\right].
\]

Now let us compute the inverse Laplace transform for the mirror part of Green's
function. Note that
\[
\frac{a_{2}+a_{1}\lambda}{a_{2}-a_{1}\lambda}\mathcal{L}\left[G\right]\left(x+y,s\right)=-\mathcal{L}\left[G\right]\left(x+y,s\right)+\frac{2a_{2}}{a_{2}-a_{1}\lambda}\mathcal{L}\left[G\right]\left(x+y,s\right).
\]
Let 
\[
g\left(x,t\right)\equiv\mathcal{L}^{-1}\left[\frac{2a_{2}}{a_{2}-a_{1}\lambda}\mathcal{L}\left[G\right]\right]\left(x,t\right),
\]
then function $g\left(x,t\right)$ satisfies 
\[
\left(a_{2}+a_{1}\partial_{x}\right)g=2a_{2}G\left(x,t\right).
\]
Solving this ODE gives 
\begin{align*}
g\left(x,t\right)=2\gamma\int_{x}^{\infty}e^{-\gamma\left(z-x\right)}G\left(z,t\right)dz=2\gamma\int_{0}^{\infty}e^{-\gamma z}G\left(x+z,t\right)dz,
\end{align*}
where $\gamma\equiv-\frac{a_{2}}{a_{1}}>0$.
Therefore 
\begin{align*}
G_{mir}\left(x+y,t\right) & =\left(-G\left(x+y,t\right)+2\gamma\int_{0}^{\infty}e^{-\gamma z}G\left(x+y+z,t\right)dz\right)\begin{pmatrix}1 & 0\\
0 & -1
\end{pmatrix}.\\
\end{align*}
By the expression of the fundamental solution  and noticing that Dirac-delta function parts
vanish for $x>0$ and $y>0$, we have 
\begin{align*}
G_{mir}\left(x+y,t\right) & =-\left(\frac{e^{-\frac{\left(x+y+ct\right)^{2}}{2\nu t}}}{\sqrt{\nu t}}\left( \begin{array}{ccc}\medskip \frac{1}{2} &\frac{1}{2c}\\
 \frac{c}{2} &\frac{1}{2} \end{array}\right)+\frac{e^{-\frac{\left(x+y-ct\right)^{2}}{2\nu t}}}{\sqrt{\nu t}}\left( \begin{array}{ccc}\medskip \frac{1}{2} &-\frac{1}{2c}\\
 -\frac{c}{2} &\frac{1}{2} \end{array}\right)\right)\\
 & \quad+2\gamma\left\{ \frac{E\left(x+y,t;-c,2\nu\right)}{\sqrt{\nu t}}\left(\begin{array}{ccc}\medskip \frac{1}{2} &\frac{1}{2c}\\
 \frac{c}{2} &\frac{1}{2} \end{array}\right)+\frac{E\left(x+y,t;c,2\nu\right)}{\sqrt{\nu t}}\left( \begin{array}{ccc}\medskip \frac{1}{2} &-\frac{1}{2c}\\
 -\frac{c}{2} &\frac{1}{2} \end{array}\right)\right\} \\
 & \quad+O\left(1\right)\left(t+1\right)^{-1/2}t^{-1/2}\left(e^{-\frac{\left(x+y-ct\right)^{2}}{C^{*}t}}+e^{-\frac{\left(x+y+ct\right)^{2}}{C^{*}t}}\right).
\end{align*}
Here  the function $E$ is defined as follows:
\begin{lemma}
\label{lem:func_E}Define function 
\[
E\left(x,t;\lambda,D_{0}\right)\equiv\int_{0}^{\infty}e^{-\gamma z}e^{-\frac{\left(x+z-\lambda t\right)^{2}}{D_{0}t}}dz.
\]
Let $\gamma>0$, $\lambda>0$. Then there exists  a constant C such that
for any given $\varepsilon>0$
\[
\frac{\partial^{k}}{\partial x^{k}}E\left(x,t;\lambda,D_{0}\right)\leq O\left(1\right)\left(t^{-k/2}e^{-\frac{\left(x-\lambda t\right)^{2}}{\left(D_{0}+\varepsilon\right)t}}+e^{-\frac{\left|x\right|+t}{C}}\right).
\]
\end{lemma}
\begin{proof}
Straightforward computation shows that
\[
E\left(x,t;\lambda,D_{0}\right)=\frac{\sqrt{\pi D_{0}t}}{2}e^{\gamma\left(x-\lambda t\right)+\frac{\gamma^{2}D_{0}t}{4}}\mbox{Erfc}\left(\frac{x-\lambda t+\frac{\gamma D_{0}t}{2}}{\sqrt{D_{0}t}}\right),
\]
where $\mbox{Erfc}(z)=\frac{2}{\sqrt{\pi}}\int_{z}^{\infty}e^{-t^2}dt$.

We consider the following cases:

Case 1: $\frac{x-\lambda t+\frac{\gamma D_{0}t}{2}}{\sqrt{D_{0}t}}\leq0$; Case 2: $0<\frac{x-\lambda t+\frac{\gamma D_{0}t}{2}}{\sqrt{D_{0}t}}<K$; Case 3: $\frac{x-\lambda t+\frac{\gamma D_{0}t}{2}}{\sqrt{D_{0}t}}\geq K$. Here the positive constant $K$ is chosen to guarantee the validity of asymptotic expansion of $\mathrm{Erfc}$ function. In fact, any $K\geq2$ will do. 

For case 1, we have
\begin{align*}
x-\lambda t =\frac{x-\lambda t}{3}+\frac{2\left(x-\lambda t\right)}{3}\leq\frac{x-\lambda t}{3}-\frac{\gamma D_{0}t}{3},
\end{align*}
which implies
\begin{align*}
E\left(x,t;\lambda,D_{0}\right)=O\left(1\right)\sqrt{t}e^{-\frac{\gamma\left|x-\lambda t\right|}{3}-\frac{\gamma^{2}D_{0}t}{12}}=O\left(1\right)e^{-\frac{\left|x\right|+t}{C}}.
\end{align*}
For case 2, 
\[
-\frac{\gamma D_{0}t}{4}<x-\lambda t+\frac{\gamma D_{0}t}{4}<-\frac{\gamma D_{0}t}{4}+K\sqrt{D_{0}t},
\]
so
\begin{align*}
E\left(x,t;\lambda,D_{0}\right) \leq O\left(1\right)\sqrt{t}e^{-\frac{\gamma^{2}D_{0}t}{4}+K\gamma\sqrt{D_{0}t}} \leq O\left(1\right)e^{-\frac{t}{C}}\leq O\left(1\right)e^{-\frac{\left|x\right|+t}{C}}.
\end{align*}
The second inequality comes from the fact that the function exponentially decays in $t$ when $t$ large and bounded when $t$ small. The last inequality is due to $x$ is bounded by $t$ for case 2, so we can sacrifice part of the time decay to gain space decay.
 
For case 3, using expansion of Erfc function
\[
\mbox{Erfc}\left(x\right)=\frac{e^{-x^{2}}}{\sqrt{\pi}x}+O\left(1\right)\frac{e^{-x^{2}}}{x^{3}},
\]
we find that 
\[
E\left(x,t;\lambda,D_{0}\right)=O\left(1\right)e^{-\frac{\left(x-\lambda t\right)^{2}}{D_{0}t}}\frac{t}{x-\lambda t+\frac{\gamma D_{0}t}{2}}.
\]
There are two subcases:

Case (3.a): $x-\lambda t\geq K\sqrt{D_{0}t}-\frac{\gamma D_{0}t}{2}$
and $\left|x-\lambda t\right|\leq\frac{\gamma D_{0}t}{4};$

Case (3.b): $x-\lambda t\geq K\sqrt{D_{0}t}-\frac{\gamma D_{0}t}{2}$
and $\left|x-\lambda t\right|>\frac{\gamma D_{0}t}{4}.$

For the first one, the numerator and denominator are comparable, hence
\[
E\left(x,t;\lambda,D_{0}\right)=O\left(1\right)e^{-\frac{\left(x-\lambda t\right)^{2}}{D_{0}t}}.
\]
For the second one, absorbing factor $t$ by exponential decay yields
that
\[
E\left(x,t;\lambda,D_{0}\right)=O\left(1\right)e^{-\frac{\left(x-\lambda t\right)^{2}}{D_{1}t}},
\]
where $D_{1}$ can be chosen as any positive number slightly larger than
$D_{0}.$ 

For derivatives we just need to absorb the extra terms into exponential
function. Thus the lemma is proved. 
\end{proof}

\begin{remark}
The function $\frac{E\left(x,t;\lambda,D_{0}\right)}{\sqrt{4\pi D_{0}t}}$
is the solution of the following equation
\begin{align*}
f_{t}+\lambda f_{x} & =D_{0}f_{xx},\\
f\left(x,0\right) & =\begin{cases}
e^{-\gamma\left|x\right|} & x\leq0,\\
0 & x>0.
\end{cases}
\end{align*}
The above lemma just states that the function behaves like the heat kernel.
\end{remark}

Summarizing,
 the Green's function of the linearized initial-boundary value problem  (\ref{1.8}) can be represented as
  \bess
  \mathbb{G}(x,t;y)=G(x-y,t)-G(x+y,t)\left( \begin{array}{ccc}\medskip  1 &0\\
 0 &-1 \end{array}\right)+g(x+y,t),  0\leq x,y<\infty,
  \eess
  and satisfies the estimates (\ref{1.8a})  given in Theorem \ref{l b}.

\begin{remark}
  The third  term in the right hand side of (\ref{1.8a}) represents the reflections wave traveling with velocity $-c$, which is resulted from the presence of boundary.
\end{remark}

\section{Nonlinear stability}
A priori energy estimate for the half space problem  for three-dimensional  case was done by  \cite{KN}, here the problem with mixed boundary condition in half line is similar,  we just omit the proof  of the existence and focus on deriving the pointwise asymptotic behavior of the nonlinear problem directly.

From Theorem \ref{l b} and Duhamel's principle applied to \eqref{eq:NL_Cauchy}, one has the integral representation of the solution
\begin{align*}
U\left(x,t\right) & =\int_{0}^{\infty}\mathbb{G}\left(x,t;y\right)U_{0}\left(y\right)dy+\int_{0}^{t}\int_{0}^{\infty}\mathbb{G}\left(x,t-s;y\right)\begin{pmatrix}0\\
Q\left(y,s\right)
\end{pmatrix}dyds\\
 & \equiv\mathcal{I}\left(x,t\right)+\mathcal{N}\left(x,t\right).
\end{align*}
where $U\left(x,t\right)=\begin{pmatrix}\rho\left(x,t\right) & m\left(x,t\right)\end{pmatrix}^{T}$.
The first term is from initial data, the second term represents the
nonlinear coupling. From the assumptions on the initial data satisfies for $\alpha=0,1$,
\[
\left|\partial_{x}^{\alpha}U_{0}\right|\leq C\varepsilon_{0}(1+x^2)^{-r}, \ \ r>1/2.
\]

The initial data gives the structure of $\partial_{x}^{\alpha}\mathcal{I}$
for $\alpha\leq1$, 
\bess
\left|\partial_{x}^{\alpha}\mathcal{I}\left(x,t\right)\right|&\le&\left|\partial_x^\alpha\int_{0}^{\infty}\mathbb{G}_{S}\left(x,t;y\right)U_{0}\left(y\right)dy\right|+\left|\int_{0}^{\infty}\partial_{x}^{\alpha}\mathbb{G}_{L}\left(x,t;y\right)U_{0}\left(y\right)dy\right|\\
&\equiv&\mathcal{I}_{1}^{\alpha}+\mathcal{I}_{2}^{\alpha}.
\eess
Here  $\mathbb{G}_S$ is the short wave parts, corresponds to singular part in Green's function $\mathbb{G}$; while $\mathbb{G}_L$ is the long wave part which dominates the large time behavior. They are given respectively by
\[
\mathbb{G}_{S}\left(x,t;0\right)=e^{-\frac{c^2t}{\nu}}(\delta(x-y)-\delta(x+y))\left( \begin{array}{ccc}\medskip 1 &0\\
0 &0 \end{array}\right),
\]
\[
\begin{aligned}
\left|\partial_{x}^{\alpha}\mathbb{G}_{L}\left(x,t;y\right)\right| & =O(1)t^{-\alpha/2}\left(\frac{e^{-\frac{(x-y+ct)^2}{2\nu t}}}{\sqrt{\nu t}}+\frac{e^{-\frac{(x-y-ct)^2}{2\nu t}}}{\sqrt{\nu t}}+\frac{e^{-\frac{(x+y-ct)^2}{(2\nu+\varepsilon)  t}}}{\sqrt{\nu t}}\right)\\
&\quad +O(1)e^{-(|x-y|+t)/C}+O(1)e^{-(|x+y|+t)/C}. 
\end{aligned}
\]
In the rest of our paper,  following the method in  \cite{D2, DW} ,  we change all $t$  in the  expression of $\mathbb{G}_{L}$ to $t+1$ to avoid the singular point  if  necessary. 

By representation of singular part $\mathbb{G}_S$ and initial condition,
\[
\mathcal{I}_{1}^{\alpha}\left(x,t\right)\leq C e^{-\frac{c^2t}{\nu}}|\partial_x^\alpha U_0(x)|\leq C\varepsilon_{0}e^{-\frac{c^2t}{\nu}}
\left(x^2+1\right)^{-r}.
\]

To estimate $\mathcal{I}_{2}^{\alpha}$, we make use  of the estimate
for regular part of Green's function $\partial_{x}^{\alpha}\mathbb{G}_{L}$.
For $\alpha=0,1,$  by Lemma \ref{lem:intial_data} with replacing $x$ by $x+ct$ or $x-ct$, we have
\begin{align*}
\mathcal{I}_{2}^{\alpha}\left(x,t\right) & =\left|\partial_{x}^{\alpha}\int_{0}^{\infty}\mathbb{G}_{L}\left(x,t;y\right)U_{0}\left(y\right)dy\right|\\
 & \leq C\varepsilon_{0}\int_{0}^{\infty}(t+1)^{-\frac{\alpha}{2}}\left[\frac{e^{-\frac{(x-y+ct)^2}{2\nu (t+1)}}}{\sqrt{\nu (t+1)}}+\frac{e^{-\frac{(x-y-ct)^2}{2\nu  (t+1)}}}{\sqrt{\nu (t+1)}}
 +\frac{e^{-\frac{(x+y-ct)^2}{(2\nu+\varepsilon) (t+1)}}}{\sqrt{\nu  (t+1)}}\right]  \left(1+y^2\right)^{-r}dy\\
 & \leq C\varepsilon_{0}\left(1+t\right)^{-\frac{\alpha}{2}} \left\{  \left(t+1+\left(|x|-c(t+1)\right)^2\right)^{-r} + \frac{e^{-\frac{(x-ct)^2}{E(t+1)}}}{\sqrt{t+1}}\right\}\,.
\end{align*}
It is easy to see  that $\mathcal{I}_{1}^{\alpha}$  has the same estimates as that of $\mathcal{I}_{2}^{\alpha}$.

Now we introduce the following notations:
\begin{align*}
\theta^{\alpha}\left(x,t;\lambda,D\right) &\equiv\left(t+1\right)^{-\alpha/2}e^{-\frac{\left[x-\lambda\left(t+1\right)\right]^{2}}{D\left(t+1\right)}},\\
\psi^{\alpha}\left(x,t;\mu\right) & \equiv \left(\sqrt{t+1}+\left|x-\mu(t+1)\right|\right)^{-\alpha}.
\end{align*}
Set 
\begin{align*}
\mathbf{A}_{0}\left(x,t\right) & =\psi^{1}\left(x,t;c\right)+\psi^{1}\left(x,t;-c\right),
\end{align*}
then there exists some positive constant $\mathscr{C}_0$ such that 
\begin{align*}
\left|\mathcal{I}\left(x,t\right)\right|\leq\mathscr{C}_{0}\varepsilon_{0}\mathbf{A}_{0}\left(x,t\right),\qquad\left|\partial_{x}\mathcal{I}\left(x,t\right)\right|\leq\mathscr{C}_{0}\varepsilon_{0}(t+1)^{-1/2}\mathbf{A}_{0}\left(x,t\right).
\end{align*}

Introduce
\begin{equation}\label{eq:ansatz}
M(T)=\sup_{0< t< T}\{\|U\mathbf{A}_{0}^{-1}\|_{L^\infty}+\|(t+1)^{1/4}U_x\mathbf{A}_{0}^{-1}\|_{L^\infty}\},
\end{equation}
then we have
\begin{equation}\label{5.14}
\left| U(x,t)\right| \leq M(t)\mathbf{A}_{0}\left(x,t\right),\quad |{U}_x(x,t)|\leq M(t)(1+t)^{-1/4}\mathbf{A}_{0}\left(x,t\right).
\end{equation}
Our goal is to show the uniform boundedness of $M(t)$ and thereby the pointwise decay of the solution as expressed in \eqref{5.14}. 
\begin{remark}
Note that there is a discrepancy between the time decay rates of the terms calculated from initial data and that of our ansatz. Actually we make this ansatz is mainly for the closure of nonlinear estimate.
\end{remark}

Now we consider the nonlinear coupling and its derivatives $\partial_{x}^{\alpha}\mathcal{N}$
for $\alpha\leq1$.  To this end, we need the estimate of nonlinear
term $Q\left(x,t\right)=\tilde{Q}_{x}$. 
Straightforward
computation shows that
\bes
&&\left|\tilde{Q}\left(x,t\right)\right| =O\left(1\right)M^{2}(t)\left[\psi^{2}\left(x,t;c\right)+\psi^{2}\left(x,t;-c\right)\right].
\lbl{5.14aa}
\ees
Now we begin to estimate the nonlinear coupling for $\alpha=0,1$.
\begin{align*}
\partial_{x}^{\alpha}\mathcal{N}\left(x,t\right) &=\int_{0}^{t}\int_{0}^{\infty}\partial_{x}^{\alpha}\mathbb{G}\left(x,t-s;y\right)
\begin{pmatrix}0\\
Q\left(y,s\right)
\end{pmatrix}dyds\\
 & =\int_{0}^{t}\int_{0}^{\infty}\partial_{x}^{\alpha}\mathbb{G}_{S}\left(x,t-s;y\right)\begin{pmatrix}0\\
Q\left(y,s\right)
\end{pmatrix}dyds+\int_{0}^{t}\int_{0}^{\infty}\partial_{x}^{\alpha}\mathbb{G}_{L}\left(x,t-s;y\right)\begin{pmatrix}0\\
\partial_{y}\tilde{Q}\left(y,s\right)
\end{pmatrix}dyds\\
 & \equiv\partial_{x}^{\alpha}\mathcal{N}_{1}+\partial_{x}^{\alpha}\mathcal{N}_{2}.
\end{align*}

From the matrix multiplication of $\mathbb{G}_S$ and $(0,Q)^T$, we have
\[
\left|\partial_{x}^{\alpha}\mathcal{N}_{1}\right|=\left|\int_{0}^{t}\int_{0}^{\infty}\partial_{x}^{\alpha}\mathbb{G}_{S}\left(x,t-s;y\right)\begin{pmatrix}0\\
Q(y,s)
\end{pmatrix}dyds\right|=0.
\]

For $\mathcal{N}_{2}$, we integrate  by parts
\begin{align*}
\left|\mathcal{N}_{2}\right|& =\left|\int_{0}^{t}\int_{0}^{\infty}\mathbb{G}_{L}\left(x,t-s;y\right)\partial_{y}\tilde{Q}\left(y,s\right)dyds\right|\\
 &\le\left|\int_{0}^{t}\int_{0}^{\infty}\partial_{y}\mathbb{G}_{L}\left(x,t-s;y\right)\tilde{Q}\left(y,s\right)dyds\right|+\left|\int_{0}^{t}\mathbb{G}_{L}\left(x,t-s;y\right)\tilde{Q}\left(y,s\right)|_{y=0}^{\infty}ds\right|\\
 & \le O\left(1\right)M^{2}(t)\int_{0}^{t}\int_{0}^{\infty}\left(t-s\right)^{-1}\left[e^{-\frac{\left(x-y-c\left(t-s\right)\right)^{2}}{2\nu\left(t-s\right)}}+e^{-\frac{\left(x-y+c\left(t-s\right)\right)^{2}}{2\nu\left(t-s\right)}}+e^{-\frac{\left(x+y-c\left(t-s\right)\right)^{2}}{(2\nu+\varepsilon)\left(t-s\right)}}\right]\\
 & \qquad\times\left[\psi^{2}\left(y,s;c\right)+\psi^{2}\left(y,s;-c\right)\right]dyds\\
 &\quad +O(1)M^{2}(t)\int_{0}^{t}\left(t-s\right)^{-1/2}\left[e^{-\frac{\left(x-c\left(t-s\right)\right)^{2}}{2\nu\left(t-s\right)}}+e^{-\frac{\left(x+c\left(t-s\right)\right)^{2}}{2\nu\left(t-s\right)}}+e^{-\frac{\left(x-c\left(t-s\right)\right)^{2}}{(2\nu+\varepsilon)\left(t-s\right)}}\right]](1+s)^{-2}ds\\
 &=\mathcal{N}_{2,1}+\mathcal{N}_{2,2}.
\end{align*}
For $\mathcal{N}_{2,1}$, we first consider
\[
\int_{0}^{t}\int_{0}^{\infty}(t-s)^{-1}e^{-\frac{\left(x-y-c\left(t-s\right)\right)^{2}}{2\nu\left(t-s\right)}}\psi^{2}\left(y,s;c\right)dyds.
\]
Noting $\psi^2(y,s;c)\leq C(t+1)^{-1/4}\psi^{3/2}(y,s;c)$, thus we  apply Lemma \ref{lem:wave-int-J-1} with $\alpha=2$, $\alpha'=0$ and $\beta=\frac{1}{2}$ to obtain it is bounded by
\[
O(1)\big[\theta^1(x,t;c,2\nu+\varepsilon)+(t+1)^{1/4}\psi^{3/2}(x,t;c)\big]\leq O(1)\psi^1(x,t;c).
\]
As for the interaction between waves with different propagation speeds, say 
\[
\int_{0}^{t}\int_{0}^{\infty}(t-s)^{-1}e^{-\frac{\left(x-y-c\left(t-s\right)\right)^{2}}{2\nu\left(t-s\right)}}\psi^{2}\left(y,s;-c\right)dyds,
\]
we apply Lemma \ref{lem:wave-int-J-2} with $\alpha=2$, $\alpha'=0$  and $\beta=\frac{1}{2}$ to dominate it by
\[
\begin{aligned}
&O(1)\theta^{3/2}(x,t;c,2\nu+\varepsilon)+ O(1)(t+1)^{1/4}\left[(x-c(t+1))^2+(t+1)^{3/2}\right]^{-3/4}\\
&+O(1)(t+1)^{1/4}\psi^1(x,t;-c)\left[(x-c(t+1))^2+(t+1)^2\right]^{-1/4}+O(1)\psi^{1/2}(x,t;c)\psi^{1/2}(x,t;-c)\\
&\leq O(1)\left[\psi^1(x,t;c)+\psi^1(x,t;-c)\right].
\end{aligned}
\]
The other terms in $\mathcal{N}_{2,1}$ can be estimated similarly, hence we have
\begin{equation}\label{eq:N21}
\mathcal{N}_{2,1}=O\left(1\right)M^{2}(t)\left[\psi^{1}\left(x,t;c\right)+\psi^{1}\left(x,t;-c\right)\right]=O\left(1\right)M^{2}(t)\mathbf{A}_{0}\left(x,t\right).
\end{equation}

For $\mathcal{N}_{2,2}$, here we only compute one of them since the others can be treated identically.
\begin{align*}
\left(I\right) &= \int_{0}^{t/2}\left.\frac{e^{-\frac{\left(x-y-c\left(t-s\right)\right)^{2}}{2\nu\left(t-s+1\right)}}}{\sqrt{\nu\left(t-s+1\right)}}\tilde{Q}(y,s)\right|_{y=0}^{\infty}ds+\int_{t/2}^{t}\left.\frac{e^{-\frac{\left(x-y-c\left(t-s\right)\right)^{2}}{2\nu\left(t-s+1\right)}}}{\sqrt{\nu\left(t-s+1\right)}}\tilde{Q}(y,s)\right|_{y=0}^{\infty}ds\\
 & =CM^{2}(t)\int_{0}^{t/2}\left(t-s+1\right)^{-1/2}e^{-\frac{\left(x-c\left(t-s\right)\right)^{2}}{2\nu\left(t-s+1\right)}}\left(1+s\right)^{-2}ds\\
 & \quad+CM^{2}(t)\int_{t/2}^{t}\left(t-s+1\right)^{-1/2}e^{-\frac{\left(x-c\left(t-s\right)\right)^{2}}{2\nu\left(t-s+1\right)}}\left(1+s\right)^{-2}ds\\
 & \equiv\left(I_{1}\right)+\left(I_{2}\right).
\end{align*}

To estimate the term $\left(I_{1}\right)$, we have two cases: $\left|x-ct\right|\leq\sqrt{t+1}$
and  $\left|x-ct\right|>\sqrt{t+1}$. 
When $\left|x-ct\right|\leq\sqrt{t+1}$, we have
\begin{align*}
\left(I_{1}\right) & \leq CM^{2}(t)\left(1+t\right)^{-1/2}\int_{0}^{t/2}\left(1+s\right)^{-2}ds \leq CM^{2}\left(\sqrt{1+t}+|x-c(t+1)|\right)^{-1}.
\end{align*}
When  $\left|x-ct\right|>\sqrt{t+1}$, we decompose the integral interval into two parts:
\[
A_{1}=\left\{ 0\leq s\leq t/2;cs\leq\left|x-ct\right|/K\right\} ,\quad A_{2}=\left\{ 0\leq s\leq t/2;cs>\left|x-ct\right|/K\right\} ,
\]
where $K$ is chosen suitably large. Then
\begin{align*}
\left(I_{1}\right) &\leq CM^{2}(t)\left(1+t\right)^{-1/2}\left\{ \int_{A_{1}}e^{-\frac{\left(x-ct\right)^{2}}{2\nu\left(t+1\right)}}\left(1+s\right)^{-2}ds+\int_{A_{2}}\left(1+\left|x-ct\right|\right)^{-1}\left(1+s\right)^{-1}ds\right\} \\
 & =CM^{2}(t)\left(1+t\right)^{-1/2}\left[e^{-\frac{\left(x-ct\right)^{2}}{D_{0}\left(t+1\right)}}+\ln(1+t)\psi^{1}\left(x,t;c\right)\right].
\end{align*}

For $\left(I_{2}\right)$, the calculation is similar. Consider two cases: $\left|x-ct\right|\leq\sqrt{t+1}$
and $\left|x-ct\right|>\sqrt{t+1}$. For $\left|x-ct\right|\leq\sqrt{t+1}$, we have
\begin{align*}
\left(I_{2}\right) & \leq CM^{2}(t)\left(1+t\right)^{-2}(1+t)^{1/2}\\
 & \le CM^{2}(t)\psi^{1}\left(x,t;c\right).
\end{align*}
For $\left|x-ct\right|>\sqrt{t+1}$, the decomposition now becomes
\bess
A_{1}=\left\{ t/2\leq\tau\leq t;c\tau\leq\left|x-ct\right|/K\right\} ,\quad A_{2}=\left\{ t/2\leq\tau\leq t;c\tau>\left|x-ct\right|/K\right\} ,
\eess
where $K$ is suitably large, which follows that
\begin{align*}
\left(I_{2}\right) & \leq CM^{2}(t)\left(1+t\right)^{-2}\int_{A_{1}}e^{-\frac{\left(x-ct\right)^{2}}{2\nu\left(t+1\right)}}\left(t-\tau\right)^{-1/2}d\tau\\
 & \quad+CM^{2}(t)\int_{A_{2}}\left(t-\tau\right)^{-1/2}\left(1+\left|x-ct\right|\right)^{-3/2}\left(1+\tau\right)^{-1/2}d\tau\\
 & =CM^{2}(t)\left[\left(1+t\right)^{-3/2}e^{-\frac{\left(x-ct\right)^{2}}{D_{0}\left(t+1\right)}}+\psi^{3/2}\left(x,t;c\right)\right].
\end{align*}
Therefore 
\[
\left(I_{1}\right)+\left(I_{2}\right)\leq CM^{2}(t)\mathbf{A}_{0}\left(x,t\right).
\]
Hence we have
\[
\mathcal{N}_{2,2}\leq C M^2(t)\mathbf{A}_{0}\left(x,t\right).
\]
Together with  \eqref{eq:N21} to yield
\[
\mathcal{N}_{2} =O\left(1\right)M^{2}(t)\mathbf{A}_{0}\left(x,t\right).
\]

For $\partial_{x}\mathcal{N}_{2}$,  we can use integration by parts to transfer the derivative in $y$ to derivative in $x$ to get 
\begin{align*}
\left|\partial_{x}\mathcal{N}_{2}\right|&=\left|\int_{0}^{t}\int_{0}^{\infty}\partial_{x}\mathbb{G}_{L}\left(x,t-s;y\right)\partial_{y}\tilde{Q}\left(y,s\right)dyds\right|\\
 &\le\left|\int_{0}^{t}\int_{0}^{\infty}\partial_{x}^2\mathbb{G}_{L}\left(x,t-s;y\right)\tilde{Q}\left(y,s\right)dyds\right|+\left|\int_{0}^{t}\partial_{x}\mathbb{G}_{L}\left(x,t-s;y\right)\tilde{Q}\left(y,s\right)|_{y=0}^{\infty}ds\right|\\
 & \equiv\mathcal{N}_{2,a}+\mathcal{N}_{2,b}.
\end{align*}
Again applying Lemma \ref{lem:wave-int-J-1} and Lemma \ref{lem:wave-int-J-2} with $\alpha=3$,  $\alpha'=0$  and $\beta=1/2$
 we
find that
\begin{align*}
\mathcal{N}_{2,a} &=O\left(1\right)M^{2}(t)\int_{0}^{t}\int_{0}^{\infty}\left(t-s\right)^{-3/2}\left[e^{-\frac{\left(x-y-c\left(t-s\right)\right)^{2}}{2\nu\left(t-s\right)}}+e^{-\frac{\left(x-y+c\left(t-s\right)\right)^{2}}{2\nu\left(t-s\right)}}+e^{-\frac{\left(x+y-c\left(t-s\right)\right)^{2}}{(2\nu+\varepsilon)\left(t-s\right)}}\right]\\
 & \quad\times\left[\psi^{2}\left(y,s;c\right)+\psi^{2}\left(y,s;-c\right)\right]dyds\\
 & \leq O\left(1\right)M^{2}(t)\sum_{\lambda=\pm c}\left\{ \theta^2(x,t;\lambda,D_0)+(t+1)^{-1/4}\log(t+1)\psi^{3/2}(x,t;\lambda)\right\}\\
 &\quad +O(1)M^2(t)\sum_{\lambda=\pm c}\left\{\theta^{5/3}(x,t;\lambda,D_0)+\psi^{3/2}(x,t;\lambda)
 \right\} \\
 & \leq O\left(1\right)M^{2}(t)(1+t)^{-1/4}\mathbf{A}_{0}\left(x,t\right).
\end{align*}
$\mathcal{N}_{2,b}$ can be treated similarly as the term $\mathcal{N}_{2,2}$ hence we omit the details here,
\[
\mathcal{N}_{2,b}  \leq O\left(1\right)M^{2}(t)(1+t)^{-1/2}\mathbf{A}_{0}\left(x,t\right).
\]
Therefore we conclude that
\[
\partial_{x}\mathcal{N}_{2}  \leq O\left(1\right)M^{2}(t)(1+t)^{-1/4}\mathbf{A}_{0}\left(x,t\right).
\]
 
Combine above estimates to obtain  
\[
M(t)\le C\varepsilon_{0}+CM^2(t),
\] 
this together with the smallness of $\varepsilon_{0}$ and the continuity of $M(t)$ lead to $M(t)
\leq C$ for $t\geq0$,
that is 
\[
|U|\leq C \mathbf{A}_{0},\quad |U_x|\leq C(t+1)^{-1/4} \mathbf{A}_{0}.
\]
This completes the proof of  Theorem \ref{LC}.

\section{Appendix}

In the Appendix, we collect some computational lemmas for wave coupling, which are used
in previous proofs.
\begin{lemma}
\label{lem:intial_data}Let $D_{0}>0$, $r>1/2$. Then for any given
$E>D_{0}$, we have
\begin{align*}
I =\int_{-\infty}^{\infty}\frac{e^{-\frac{\left(x-y\right)^{2}}{D_{0}\left(t+1\right)}}}{\sqrt{t+1}}\left(1+y^2\right)^{-r}dy
  =O\left(1\right)\left[\frac{e^{-\frac{x^{2}}{E\left(t+1\right)}}}{\sqrt{t+1}}+\left(t+1+x^2\right)^{-r}\right].
\end{align*}
\end{lemma}
\begin{proof}
We consider the following two cases:

Case 1. $\left|x\right|\leq\sqrt{t+1}$;

Case 2. $\left|x\right|>\sqrt{t+1}$.

For case 1, one has
\[
I\leq\int_{-\infty}^{\infty}\frac{1}{\sqrt{t+1}}\left(1+y^2\right)^{-r}dy\leq\frac{O\left(1\right)}{\sqrt{t+1}}.
\]

For case 2, we decompose the integration region into two parts,
\begin{align*}
I & =\int_{\left|y\right|<\frac{\left|x\right|}{M}}\frac{e^{-\frac{\left(x-y\right)^{2}}{D_{0}\left(t+1\right)}}}{\sqrt{t+1}}\left(1+y^2\right)^{-r}dy+\int_{\left|y\right|>\frac{\left|x\right|}{M}}\frac{e^{-\frac{\left(x-y\right)^{2}}{D_{0}\left(t+1\right)}}}{\sqrt{t+1}}\left(1+y^2\right)^{-r}dy\\
 & \leq O\left(1\right)\int_{\left|y\right|<\frac{\left|x\right|}{M}}\frac{e^{-\frac{x^{2}}{E\left(t+1\right)}}}{\sqrt{t+1}}\left(1+y^2\right)^{-r}dy+O\left(1\right)\int_{\left|y\right|>\frac{\left|x\right|}{M}}\frac{e^{-\frac{\left(x-y\right)^{2}}{D_{0}\left(t+1\right)}}}{\sqrt{t+1}}(1+x^2)^{-r}dy\\
 & \leq O\left(1\right)\left[\frac{e^{-\frac{x^{2}}{E\left(t+1\right)}}}{\sqrt{t+1}}+\left(t+1+x^2\right)^{-r}\right].
\end{align*}
This completes the proof.
\end{proof}
\begin{lemma}[\cite{LZ1999}]
\label{lem:wave-int-J-1}Let $\alpha\geq\alpha'\geq0$, $\alpha-\alpha'<3$,
$\beta\geq0$, $\nu>0$ and $\lambda$ be a constant. Then for any
given $\varepsilon>0$, and all $-\infty<x<\infty$, $t\geq0$, we
have

\[
\begin{aligned} & \int_{0}^{t}\int_{-\infty}^{\infty}\left(t-s\right)^{-\left(\alpha-\alpha'\right)/2}\left(t-s+1\right)^{-\alpha'/2}e^{-\frac{\left(x-y-\lambda\left(t-s\right)\right)^{2}}{\nu\left(t-s\right)}}\cdot\left(s+1\right)^{-\beta/2}\psi^{3/2}\left(y,s;\lambda\right)dyds\\
 & =O\left(1\right)\left[\theta^{\gamma}\left(x,t;\lambda,\nu+\varepsilon\right)+\left(t+1\right)^{-\sigma/2}\psi^{3/2}\left(x,t;\lambda\right)\right]\\
 & \quad+\begin{cases}
O\left(1\right)\theta^{\gamma}\left(x,t;\lambda,\nu+\varepsilon\right)\log\left(t+1\right) & \mbox{for }\beta=3/2,\\
0, & otherwise,
\end{cases}\\
 & \quad+\begin{cases}
O\left(1\right)\left(t+1\right)^{-\sigma/2}\psi^{3/2}\left(x,t;\lambda\right)\log\left(t+1\right) & \mbox{for }\alpha=3\mbox{ or }\beta=2,\\
0, & otherwise,
\end{cases}
\end{aligned}
\]
where $\gamma=\alpha+\min\left(\beta,\frac{3}{2}\right)-\frac{3}{2}$,
and $\sigma=\min\left(\alpha,3\right)+\min\left(\beta,2\right)-3$.
\end{lemma}
\begin{lemma}[\cite{LZ1999}]
\label{lem:wave-int-J-2}Let the constants $\alpha\geq1$, $\alpha'\geq0$,
$0\leq\alpha-\alpha'<3$, $\beta\geq0$, $\nu>0$ and $\lambda\neq\lambda'$.
Then for any given $\varepsilon>0$, $K>2\left|\lambda-\lambda'\right|$,
and all $-\infty<x<\infty$, $t\geq0$, we have

\[
\begin{aligned} & \int_{0}^{t}\int_{-\infty}^{\infty}\left(t-s\right)^{-\left(\alpha-\alpha'\right)/2}\left(t-s+1\right)^{-\alpha'/2}e^{-\frac{\left(x-y-\lambda\left(t-s\right)\right)^{2}}{\nu\left(t-s\right)}}\cdot\left(s+1\right)^{-\beta/2}\psi^{3/2}\left(y,s;\lambda'\right)dyds\\
 & =O\left(1\right)\theta^{\gamma}\left(x,t;\lambda,\nu+\varepsilon\right)+O\left(1\right)\left(t+1\right)^{-\sigma/2}\cdot\left[\left(x-\lambda\left(t+1\right)\right)^{2}+\left(t+1\right)^{\left(5/3\right)-\left(1/3\right)\min\left(\beta,2\right)}\right]^{-3/4}\\
 & \quad+O\left(1\right)\left(t+1\right)^{-\sigma'/2}\left(\psi^{3/2}\left(x,t;\lambda'\right)\right)^{\left(1/3\right)\min\left(\alpha,3\right)}\\
 & \quad\cdot\left[\left(x-\lambda\left(t+1\right)\right)^{2}+\left(t+1\right)^{2}\right]^{-\left(3/4\right)\left(1-\left(1/3\right)\min\left(\alpha,3\right)\right)}\cdot\begin{cases}
1 & \mbox{if }\alpha\neq3\\
1+\log\left(t+1\right) & \mbox{if }\alpha=3
\end{cases}\\
 & \quad+O\left(1\right)\left|x-\lambda\left(t+1\right)\right|^{-\left(1/2\right)\min\left(\beta,5/2\right)-1/4}\left|x-\lambda'\left(t+1\right)\right|^{-\left(1/2\right)\left(\alpha-1\right)}\\
 & \quad\cdot\mbox{char}\left\{ \min\left(\lambda,\lambda'\right)\left(t+1\right)+K\sqrt{t+1}\leq x\leq\max\left(\lambda,\lambda'\right)\left(t+1\right)-K\sqrt{t+1}\right\} \\
 & \quad+O\left(1\right)\begin{cases}
\theta^{\alpha}\left(x,t;\lambda,\nu+\varepsilon\right)\log\left(t+1\right) & \mbox{if }\beta=3/2,\\
\left(t+1\right)^{-\left(1/2\right)\left(\alpha-1\right)}\psi^{3/2}\left(x,t;\lambda\right)\log\left(t+1\right) & \mbox{if }\beta=2,\\
0, & otherwise,
\end{cases}
\end{aligned}
\]
where $\gamma=\alpha+\frac{1}{2}\min\left(\beta,\frac{3}{2}\right)-\frac{3}{4}$,
$\sigma=\alpha+\min\left(\beta,2\right)-3$, $\sigma'=\min\left(\alpha,3\right)+\beta-3$.
\end{lemma}

\section*{Acknowledgments}
The authors want to thank Professor Tai-Ping Liu for the discussions on the nonlinear decay rates, and thank anonymous referee for helpful comments and suggestions which improve the presentation of paper significantly. Du was supported by NSFC(Grant No. 11526049 and 11671075) and the Fundamental Research Funds for the Central Universities (No. 17D110906). Wang wants to thank Institute of Mathematics, Academia Sinica for financial support. Part of the work was done while Wang was a postdoc there.

\end{document}